\documentclass[11pt,a4paper]{article}

\usepackage[T1]{fontenc}
\usepackage[utf8]{inputenc}
\usepackage{amsmath,amssymb,amsthm}
\usepackage{mathtools}
\usepackage{booktabs}
\usepackage{geometry}
\usepackage{enumitem}
\usepackage{xcolor}
\usepackage[colorlinks=true,linkcolor=blue!55!black,citecolor=blue!55!black,
            urlcolor=blue!55!black]{hyperref}

\hypersetup{
  pdftitle={The S-function at arbitrary order: hyperplanes of symmetries and
            the realisability condition},
  pdfauthor={L. G. S. Duarte and L. A. C. P. da Mota}
}

\newtheorem{theorem}{Theorem}[section]
\newtheorem{lemma}[theorem]{Lemma}
\newtheorem{proposition}[theorem]{Proposition}
\newtheorem{corollary}[theorem]{Corollary}
\theoremstyle{definition}
\newtheorem{definition}[theorem]{Definition}
\newtheorem{example}[theorem]{Example}
\theoremstyle{remark}

\newcommand{\D}{\mathcal{D}}
\newcommand{\CC}{\mathcal{C}}
\newcommand{\VV}{\mathcal{V}}
\newcommand{\PP}{\mathbb{P}}
\newcommand{\Gr}{\operatorname{Gr}}
\newcommand{\jj}{\mathbf{j}}
\newcommand{\vv}{\mathbf{v}}
\newcommand{\Ef}{\mathfrak{E}}
\newcommand{\Lie}{\mathcal{L}}
\newcommand{\todo}[1]{\textcolor{red}{\textbf{[#1]}}}

\title{\textbf{The $S$-function at arbitrary order:\\
hyperplanes of symmetries and the realisability condition}}

\author{%
L.\,G.\,S.~Duarte\thanks{\texttt{lgsduarte@gmail.com}} \ and L.\,A.\,C.\,P.~da~Mota\thanks{Corresponding author. \texttt{lacpdamota@uerj.br}}\\[6pt]
\small Departamento de F\'{\i}sica Te\'orica, Instituto de F\'{\i}sica,\\
\small Universidade do Estado do Rio de Janeiro,\\
\small Rua S\~ao Francisco Xavier 524, 20550-013 Rio de Janeiro, RJ, Brazil}

\date{}

\begin{document}
\maketitle

\begin{abstract}
For a scalar equation $y^{(n)}=\varphi(x,y,\dots,y^{(n-1)})$ the $S$-function of
the extended Prelle--Singer method is a vector $S=(S_0,\dots,S_{n-2})$, the
null-form vector of Chandrasekar, Senthilvelan, Lakshmanan and collaborators,
and the jet gradient of a first integral solves the adjoint of the variational
equation.  We describe the solutions of the system satisfied by $S$
intrinsically, and determine which of them come from first integrals.  Over the
field $\CC$ of first integrals the solutions are exactly the hyperplanes of the
$n$-dimensional space $\VV$ of symmetry characteristics, recovered from $n-1$
independent characteristics by a Plücker formula; in flow-box coordinates they
are the hyperplane fields on the local space of leaves.  A solution is the
$S$-vector of a first integral if and only if the corresponding hyperplane
field is integrable.  This Frobenius condition splits into the $S$-system
itself plus $\binom n3$ further differential conditions, so it is vacuous for
$n=2$ and a genuine restriction for every $n\ge3$, for which we give an
explicit non-realisable solution.  The same picture explains why
$\lambda$-symmetries and $S$-functions coincide only when $n=2$: line fields and
hyperplane fields agree only on surfaces.  For the computation of $S$ we
record that the system is equivalent to one scalar equation for the last
component, which is the Riccati form of the adjoint variational equation, and
we spell out what this does and does not imply for a rational search.
\end{abstract}



\section{Introduction}\label{sec:intro}

The extension of the Prelle--Singer procedure~\cite{PrelleSinger1983} to
second-order equations~\cite{DuarteDuarteDaMota2001,DDMS2002} replaces the
search for a first integral $I$ of $y''=\varphi(x,y,y')$ by the search for the
single function $S=I_y/I_{y'}$, which satisfies a Riccati equation along the
flow.  The method has been developed into an effective tool for rational
equations~\cite{ACDM2019,DEM2024,NewS2023}.  The same object was reached
independently as the \emph{null form} of Chandrasekar, Senthilvelan and
Lakshmanan~\cite{CSL2005,CSL2005n,CSL2006} and, up to sign, as the $\lambda$ of
the $\lambda$-symmetries of Muriel and Romero~\cite{MurielRomero2001,MurielRomero2009}.

At order $n$ the $S$-function becomes a vector.  Most of the elementary
structure is already in print, and we state precisely what, because it
delimits what is new here.
\begin{itemize}[leftmargin=*,itemsep=2pt]
\item The system satisfied by the null-form vector at general order, with its
compatibility conditions, is in~\cite{CSL2005n}.  For $n=3$ the system was
reduced to a single equation for one component in~\cite[eqs.~(3.39)--(3.40)]{CSL2006}.
\item For $n=2$, the linearisation $S=-\D(X)/X$ of the $S$-equation into the
variational equation, the identity $\lambda=-S$, and the relation between
integrating factor, characteristic and Jacobi multiplier are
in~\cite{Mohanasubha2014}.
\item For general $n$, the null-form vector normalised by its last component,
its linearisation, the adjoint variational equation with the statement that
the integrating factor is an adjoint symmetry, and the relation
$\sum_kS_k\D^k(Q)=0$ between the null-form vector and symmetry characteristics
$Q$ are in~\cite{Mohanasubha2016}; that a single $\lambda$ does not fix the
null forms when $n=3$ is observed in~\cite{Mohanasubha2015}.
\item That adjoint symmetries and integrating factors are related by an
additional condition is the theme of~\cite{AncoBluman1998}.
\end{itemize}

What we add is a description of the \emph{solution set} of the $S$-system and
of the subset coming from first integrals.

\begin{enumerate}[label=\textbf{(R\arabic*)},leftmargin=*,itemsep=4pt]
\item \textbf{Solutions are hyperplanes.}  Over the field $\CC$ of first
integrals, the solutions of the $S$-system are in bijection with the
hyperplanes of the $n$-dimensional $\CC$-space $\VV$ of symmetry
characteristics, $\{\text{solutions}\}=\PP(\VV^\vee)$, and $n-1$ independent
characteristics give the solution by a Plücker formula
(Theorem~\ref{thm:plucker}).
\item \textbf{Realisability.}  A solution is the $S$-vector of a first
integral if and only if a Frobenius condition holds.  Given the $S$-system,
this condition reduces to the vanishing of $\binom n3$ components; in
flow-box coordinates the solutions are the hyperplane fields on the local leaf
space and the realisable ones are the integrable ones
(Theorems~\ref{thm:frobenius} and~\ref{thm:intrinsic}).  The inclusion of the
realisable solutions is strict for every $n\ge3$
(Corollary~\ref{cor:strict}).
\item \textbf{The two accidents of order two.}  On a surface every
hyperplane field is integrable, and line fields and hyperplane fields agree.
These are the reasons why every solution of the $S$-equation comes from a
first integral when $n=2$, and why $\lambda=-S$ holds only then
(Corollary~\ref{cor:accidents}).
\end{enumerate}

For the search we record, in Theorem~\ref{thm:collapse}, that the $S$-system is
equivalent to one scalar equation $\Ef_n(\sigma)=0$ for the last component
$\sigma$, and that this equation is the Riccati form of the adjoint
variational equation $L^*(R)=0$ under $\sigma=-\D(R)/R-\varphi_{n-1}$; the case
$n=3$ is~\cite[(3.39)--(3.40)]{CSL2006}.  The practical content is that a
rational solution of the $S$-system corresponds to an adjoint symmetry $R$ with
rational logarithmic derivative, and that a rational ansatz is an ansatz for one
function.
Section~\ref{sec:method} states the consequences for the method, including the
facts that a solution of $\Ef_n=0$ must be tested for realisability when
$n\ge3$, and that recovering $I$ from $\sigma$ is in general an
integrating-factor problem rather than a pair of quadratures.

\paragraph{Structure.}  Section~\ref{sec:setting} fixes the setting and recalls
$n=2$.  Section~\ref{sec:adjoint} derives the $S$-system and its scalar
form.  Section~\ref{sec:grassmann} proves (R1), Section~\ref{sec:frobenius}
proves (R2) and (R3), Section~\ref{sec:method} draws the consequences for the
method, Section~\ref{sec:examples} works the examples, and
Section~\ref{sec:outlook} lists open problems.  Appendix~\ref{app:closing}
records $\Ef_n$ for $n\le4$ and the shape of its coefficients.

\section{Setting}\label{sec:setting}

\subsection{The field of first integrals}\label{ssec:field}

Let
\begin{equation}\label{eq:nODE}
  y^{(n)} = \varphi(x,z_0,z_1,\dots,z_{n-1}),\qquad z_k\equiv y^{(k)},
\end{equation}
with $\varphi$ holomorphic (in applications, rational) on a domain of
$\mathbb{C}^{n+1}$, and let
\begin{equation}\label{eq:Dn}
  \D = \partial_x + \sum_{k=0}^{n-2}z_{k+1}\,\partial_{z_k}
       + \varphi\,\partial_{z_{n-1}},\qquad
  \vv=(1,z_1,\dots,z_{n-1},\varphi),
\end{equation}
be the total derivative along~\eqref{eq:nODE}.  We write
$\varphi_k:=\partial\varphi/\partial z_k$.

Throughout, $U$ is a flow box: there are holomorphic first integrals
$I_1,\dots,I_n$ on $U$ such that $(x,I_1,\dots,I_n)$ are coordinates on $U$, in
which $\D=\partial_x$ and $U\cong J\times N$ with $J\subset\mathbb{C}$ and
$N\subset\mathbb{C}^n$ connected open sets.  All functions are meromorphic on $U$.  The
meromorphic functions on $U$ form a field $\mathcal{M}$, $\D$ is a derivation
of $\mathcal{M}$, and
\[
  \CC:=\ker\D=\{\,F(I_1,\dots,I_n)\,\}
\]
is a subfield, the \emph{field of first integrals}.  Every statement below is
local in this sense.  (In the smooth category $\ker\D$ is only a ring, and the
Wronskian criterion used below fails; this is why we work analytically.)

\begin{definition}\label{def:var}
The \emph{variational equation} of~\eqref{eq:nODE} is
\begin{equation}\label{eq:varn}
  L(Q) := \D^n(Q) - \sum_{k=0}^{n-1}\varphi_k\,\D^k(Q) = 0,
\end{equation}
and $\VV:=\{Q\in\mathcal{M}:L(Q)=0\}$, a $\CC$-vector space.  Writing
$\jj(Q)=(Q,\D Q,\dots,\D^{n-1}Q)^\top$, $L(Q)=0$ is $\D\jj(Q)=B\,\jj(Q)$ with
the companion matrix $B$: $B_{k,k+1}=1$ $(k\le n-2)$, $B_{n-1,k}=\varphi_k$.
\end{definition}

Characteristics $Q=\eta-\xi y'$ of Lie point symmetries, and more generally of
evolutionary symmetries, lie in $\VV$.

\begin{lemma}\label{lem:dimV}
$\dim_\CC\VV=n$.  Moreover $Q_1,\dots,Q_m\in\VV$ are $\CC$-linearly independent
if and only if the $m\times n$ matrix $[\D^j(Q_i)]$ has rank $m$.
\end{lemma}
\begin{proof}
Let $V$ be the matrix $V_{ik}=\partial I_i/\partial z_k$, which is invertible on
$U$ since $\det V$ is the Jacobian of the change of coordinates.  By
Lemma~\ref{lem:vk} below (whose proof does not use this lemma), $\D(V)=-VB$, so $W:=V^{-1}$ satisfies
$\D(W)=BW$.  Because $B$ is a companion matrix, each column of $W$ is
$\jj(Q^*_j)$ for its top entry $Q^*_j$, and $Q^*_j\in\VV$.  If $u$ solves
$\D u=Bu$ then $c:=Vu$ satisfies $\D(c)=-VBu+VBu=0$, so $c\in\CC^n$ and
$u=Wc$; hence $Q^*_1,\dots,Q^*_n$ is a basis of $\VV$.  For the rank
statement: a nontrivial $\CC$-relation $\sum c_iQ_i=0$ gives the same relation
among the rows $\jj(Q_i)^\top$; conversely, if $Q_1,\dots,Q_m$ are
$\CC$-independent, their $m\times m$ Wronskian is nonzero by the Wronskian
criterion over the differential field $(\mathcal{M},\D)$, whose field of
constants is $\CC$~\cite[Lemma~1.12]{vdPutSinger}.
\end{proof}

\subsection{The case \texorpdfstring{$n=2$}{n=2}}\label{ssec:two}

For $y''=\varphi(x,y,z)$ and a first integral $I$ with $I_z\ne0$ put $R=I_z$
and $S=I_y/I_z$.  Differentiating $\D(I)=0$ gives
\begin{equation}\label{eq:DS}
  \D(S) = S^2+\varphi_zS-\varphi_y,\qquad \D(R)=-(S+\varphi_z)R .
\end{equation}
The substitution $S=-\D(Q)/Q$ turns the first equation into the variational
equation $\D^2Q-\varphi_z\D Q-\varphi_yQ=0$~\cite[eqs.~(10)--(11)]{Mohanasubha2014},
so the solutions of the $S$-equation are $\PP(\VV)$: $Q\mapsto-\D(Q)/Q$ is
constant exactly on the orbits of $\CC^\times$, and every solution arises
from some $Q$ (Theorem~\ref{thm:plucker} with $n=2$).  With $\VV=\CC Q_1\oplus\CC Q_2$ and $S_i=-\D(Q_i)/Q_i$,
every solution is
\[
  S=\frac{c_1Q_1S_1+c_2Q_2S_2}{c_1Q_1+c_2Q_2},\qquad [c_1:c_2]\in\PP^1(\CC),
\]
so the cross-ratio of four solutions lies in $\CC$.  The $\lambda$-symmetries
of the equation with respect to $\partial_y$ are $\lambda=-S$~\cite{MurielRomero2001,Mohanasubha2014},
and for two independent first integrals
$\partial(I_1,I_2)/\partial(y,z)=R_1R_2(S_1-S_2)$ is a Jacobi last multiplier.
None of this is new; the question of this paper is which of these features
persist when $n\ge3$.

\section{The adjoint system and its scalar form}\label{sec:adjoint}

\begin{lemma}\label{lem:vk}
Let $I$ be a first integral of~\eqref{eq:nODE} and $v_k:=I_{z_k}$
$(k=0,\dots,n-1)$, $v_{-1}:=0$.  Then
\begin{equation}\label{eq:vsys}
  \D(v_k) = -v_{k-1} - \varphi_k\,v_{n-1},\qquad k=0,\dots,n-1,
\end{equation}
that is, $\D(v)=-B^\top v$ for $v=\nabla_zI$: the jet gradient solves the
adjoint of the variational equation.  Consequently, for every $Q\in\VV$,
\begin{equation}\label{eq:pairing}
  \Pi(I,Q):=\sum_{k=0}^{n-1}I_{z_k}\,\D^k(Q)\in\CC,\qquad
  \text{and in general}\quad \D\,\Pi(I,Q)=I_{z_{n-1}}\,L(Q).
\end{equation}
\end{lemma}
\begin{proof}
Differentiate $I_x+\sum_{j\le n-2}z_{j+1}I_{z_j}+\varphi I_{z_{n-1}}=0$ with
respect to $z_k$; the terms not of the form $\D(I_{z_k})$ are
$I_{z_{k-1}}$ (absent for $k=0$) and $\varphi_kI_{z_{n-1}}$.
For~\eqref{eq:pairing},
$\D\Pi=\sum_k(-v_{k-1}-\varphi_kv_{n-1})\D^kQ+\sum_kv_k\D^{k+1}Q
=v_{n-1}L(Q)$, the other terms cancelling in pairs.
\end{proof}

Identity~\eqref{eq:pairing} is the Lagrange identity for the pair
$(L,L^*)$; that adjoint symmetries and symmetries pair to first integrals is
classical (see~\cite{AncoBluman1998} and~\cite[\S2]{Mohanasubha2016}).

\begin{definition}\label{def:Svector}
For a first integral $I$ with $I_{z_{n-1}}\ne0$, its \emph{integrating
factor} and \emph{$S$-vector} are
\begin{equation}\label{eq:RS}
  R := I_{z_{n-1}},\qquad
  S_k := I_{z_k}/R\quad(k=0,\dots,n-2),
\end{equation}
with $S_{n-1}:=1$, $S_{-1}:=0$ and \emph{principal component}
$\sigma:=S_{n-2}$.
\end{definition}

In~\cite{Mohanasubha2016} the null forms are defined by
$I_x=-RS_1$, $I_{x^{(i)}}=-RS_{i+1}$, $I_{x^{(n-1)}}=-R$; after the index shift
the $S$ agree with ours and only $R$ changes sign.

\begin{proposition}[The $S$-system]\label{prop:Ssys}
With $R$, $S_k$ as in~\eqref{eq:RS},
\begin{align}
  \D(R)   &= -(\sigma+\varphi_{n-1})\,R, \label{eq:DRn}\\
  \D(S_k) &= S_k\,\sigma + \varphi_{n-1}S_k - S_{k-1} - \varphi_k,
             \qquad k=0,\dots,n-2. \label{eq:Ssys}
\end{align}
\end{proposition}
\begin{proof}
\eqref{eq:DRn} is~\eqref{eq:vsys} for $k=n-1$ divided by $R$.  Then
$\D(S_k)=\D(v_k)/R-S_k\D(R)/R=-S_{k-1}-\varphi_k+S_k(\sigma+\varphi_{n-1})$.
\end{proof}

For $n=2$, \eqref{eq:Ssys} is the $S$-equation~\eqref{eq:DS}.  Read upwards
from $S_{n-1}=1$, the system determines every component from $\sigma$.

\begin{theorem}[Scalar form]\label{thm:collapse}
Define $\widehat S_{n-1}:=1$ and, for $k=n-1,\dots,0$,
\begin{equation}\label{eq:recursion}
  \widehat S_{k-1} := -\D\bigl(\widehat S_k\bigr)
     + \widehat S_k\,\sigma + \varphi_{n-1}\widehat S_k - \varphi_k .
\end{equation}
\begin{enumerate}[label=\textup{(\roman*)},nosep,leftmargin=*]
\item $\widehat S_{n-2}=\sigma$, and $\widehat S_k$ is a differential polynomial
in $\sigma$ of degree $n-1-k$ and order $n-2-k$.
\item $(S_0,\dots,S_{n-2})$ solves~\eqref{eq:Ssys} if and only if
$S_k=\widehat S_k$ for all $k$ and
\begin{equation}\label{eq:closing}
  \Ef_n(\sigma) := \widehat S_{-1} = 0,
\end{equation}
an equation of order $n-1$ and degree $n$ in $\sigma$, quasi-linear in
$\D^{n-1}\sigma$.
\item If $R\ne0$ and $\sigma=-\D(R)/R-\varphi_{n-1}$, then
$\widehat S_kR=v_k$, where $v_{n-1}=R$ and $v_{k-1}=-\D(v_k)-\varphi_kR$, and
\begin{equation}\label{eq:adjointRiccati}
  R\,\Ef_n(\sigma)=L^*(R):=(-1)^n\D^nR-\sum_{k=0}^{n-1}(-1)^k\D^k(\varphi_kR).
\end{equation}
Thus $\Ef_n(\sigma)=0$ is the Riccati form of the adjoint variational
equation $L^*(R)=0$.
\end{enumerate}
\end{theorem}
\begin{proof}
(i) and (ii): the case $k=n-1$ of~\eqref{eq:recursion} gives
$\widehat S_{n-2}=\sigma$, and~\eqref{eq:recursion} is~\eqref{eq:Ssys} solved
for $S_{k-1}$; the convention $S_{-1}=0$ is~\eqref{eq:closing}.  Degrees,
orders and quasi-linearity are Proposition~\ref{prop:shape}.
(iii) Multiplying~\eqref{eq:recursion} by $R$ and using
$\D(R)=-(\sigma+\varphi_{n-1})R$ gives
$\widehat S_{k-1}R=-\D(\widehat S_kR)-\varphi_kR$, so $v_k:=\widehat S_kR$
obeys the stated recursion, which is~\eqref{eq:vsys}.  Unwinding it from
$v_{n-1}=R$ gives $v_{-1}=(-\D)^nR-\sum_k(-\D)^k(\varphi_kR)=L^*(R)$.
\end{proof}

For $n=3$, \eqref{eq:closing} reads
\begin{equation}\label{eq:E3}
  \D^2\sigma = 3\sigma\D\sigma + 2\varphi_2\D\sigma - \sigma^3 - 2\varphi_2\sigma^2
      + \bigl(\varphi_1-\varphi_2^2+\D\varphi_2\bigr)\sigma
      + \varphi_0 + \varphi_1\varphi_2 - \D\varphi_1 ,
\end{equation}
which is~\cite[eq.~(3.40)]{CSL2006} with their $U$ equal to our $\sigma$.
(The term $-2\varphi_2\D U$ is missing from~(3.40) in the arXiv version
of~\cite{CSL2006}; substituting their~(3.39) into their~(2.6) restores it.)  The linearisation of the general-$n$ system used
in~\cite[eqs.~(2.4)--(2.5)]{Mohanasubha2016} is equivalent to~(iii).

\section{Solutions of the \texorpdfstring{$S$}{S}-system are hyperplanes}\label{sec:grassmann}

\begin{lemma}[Gauge invariance]\label{lem:gauge}
If $u=(u_0,\dots,u_{n-1})$ satisfies $\D(u)=-B^\top u+g\,u$ for some function
$g$, with $u_{n-1}\ne0$, then $S_k:=u_k/u_{n-1}$ solves~\eqref{eq:Ssys}.
\end{lemma}
\begin{proof}
The hypothesis reads $\D(u_k)=-u_{k-1}-\varphi_ku_{n-1}+gu_k$.  With
$\sigma=u_{n-2}/u_{n-1}$,
$\D(S_k)=(-S_{k-1}-\varphi_k+gS_k)-S_k(-\sigma-\varphi_{n-1}+g)$,
which is~\eqref{eq:Ssys}.
\end{proof}

\begin{theorem}[Plücker description]\label{thm:plucker}
Let $Q_1,\dots,Q_{n-1}\in\VV$ be $\CC$-independent, let $\Theta$ be the
$(n-1)\times n$ matrix $\Theta_{ij}=\D^j(Q_i)$ $(j=0,\dots,n-1)$, and let
$m_k$ be the minor of $\Theta$ obtained by deleting column $k$.  Then
$m_{n-1}\not\equiv0$, and off its zero set
\begin{equation}\label{eq:plucker}
  S_k := (-1)^{k+n-1}\,\frac{m_k}{m_{n-1}},\qquad k=0,\dots,n-2,
\end{equation}
solves~\eqref{eq:Ssys} and satisfies
\begin{equation}\label{eq:annihilate}
  \sum_{k=0}^{n-1}S_k\,\D^k(Q)=0\qquad\text{for } Q\in H:=\langle
  Q_1,\dots,Q_{n-1}\rangle_\CC .
\end{equation}
Conversely every solution of~\eqref{eq:Ssys} arises in this way from exactly
one hyperplane $H\subset\VV$.  Hence
\begin{equation}\label{eq:GrS}
  \{\,\text{solutions of~\eqref{eq:Ssys}}\,\}\ =\ \PP(\VV^\vee)
  \ \cong\ \Gr(n-1,\VV).
\end{equation}
\end{theorem}
\begin{proof}
$m_{n-1}$ is the $\D$-Wronskian of $Q_1,\dots,Q_{n-1}$, which is not
identically zero by Lemma~\ref{lem:dimV}.  Let $u_k:=(-1)^km_k$.  For each
$i$, $\sum_ku_k\D^k(Q_i)$ is the determinant of $\Theta$ with the row
$\jj(Q_i)^\top$ added, hence zero; this gives~\eqref{eq:annihilate}.  The rows
of $\Theta$ evolve by $\D\jj(Q_i)=B\jj(Q_i)$, so the vector $u$ of signed
maximal minors evolves by the induced derivation on $\Lambda^{n-1}$, which in
these coordinates is $(\operatorname{tr}B)\,\mathrm{Id}-B^\top$:
\[
  \D(u)=-B^\top u+\varphi_{n-1}u .
\]
Lemma~\ref{lem:gauge} then gives~\eqref{eq:Ssys}.

Conversely let $S$ solve~\eqref{eq:Ssys}, put $u_k:=S_k$, $u_{n-1}:=1$, and
$g:=\sigma+\varphi_{n-1}$; then $\D(u)=-B^\top u+gu$.  For
$p(Q):=\sum_ku_k\D^k(Q)$ the computation of Lemma~\ref{lem:vk} gives
$\D\,p(Q)=g\,p(Q)+u_{n-1}L(Q)=g\,p(Q)$ on $\VV$.  The map $p:\VV\to\mathcal{M}$
is $\CC$-linear.  It is not zero: for a basis $Q^*_1,\dots,Q^*_n$ of $\VV$,
$(p(Q^*_j))_j=W^\top u$ with $W$ the invertible matrix of
Lemma~\ref{lem:dimV}, and $u\ne0$.  If $p(Q_0)\ne0$ then
$\D\bigl(p(Q)/p(Q_0)\bigr)=0$, so $p(\VV)=\CC\,p(Q_0)$ and $H:=\ker p$ is a
hyperplane.  Applying the first part to a basis of $H$ returns a covector
annihilating $H$ and normalised by $u_{n-1}=1$, hence $S$.  Two hyperplanes
giving the same $S$ have the same annihilator, so they coincide.
\end{proof}

For $n=2$, \eqref{eq:plucker} is $S_0=-\D(Q)/Q$ and~\eqref{eq:GrS} is the
statement of Section~\ref{ssec:two}.  The chart $S_{n-1}=1$ excludes no point
of $\Gr(n-1,\VV)$, since $m_{n-1}\not\equiv0$ for every hyperplane; it is the
zero set of $m_{n-1}$ (equivalently of $R$) inside $U$ that is excluded.

\begin{proposition}\label{prop:Rnonzero}
If $I$ is a first integral with $I_{z_{n-1}}\equiv0$, then $I$ is constant.
\end{proposition}
\begin{proof}
If $v_{n-1}\equiv0$, \eqref{eq:vsys} for $k=n-1$ gives $v_{n-2}\equiv0$, and
descending $v_k\equiv0$ for all $k$; then $\D(I)=0$ gives $I_x=0$.
\end{proof}

\section{Realisability}\label{sec:frobenius}

Theorem~\ref{thm:plucker} describes all solutions of the $S$-system; it does
not say that they are $S$-vectors of first integrals.

\begin{definition}\label{def:theta}
For any meromorphic $(S_0,\dots,S_{n-2})$, with $S_{n-1}:=1$, let
\begin{equation}\label{eq:theta}
  \theta_S := \sum_{k=0}^{n-1}S_k\,\mathrm{d}z_k
   \;-\;\Bigl(\sum_{j=0}^{n-2}z_{j+1}S_j+\varphi\Bigr)\mathrm{d}x ,
\end{equation}
the unique $1$-form with these $z$-components and $\theta_S(\vv)=0$.  We call
$S$ \emph{realisable} if, near every point of a dense open subset of $U$, it is
the $S$-vector of a first integral.
\end{definition}

If $S$ is the $S$-vector of $I$ then $\theta_S=\mathrm{d}I/R$.

\begin{lemma}\label{lem:Lie}
For any $S$, the Lie derivative of $\theta_S$ along $\vv$ is
\[
  (\Lie_\vv\theta_S)_{z_k}=\D(S_k)+S_{k-1}+\varphi_k\ \ (k\le n-2),\qquad
  (\Lie_\vv\theta_S)_{z_{n-1}}=\sigma+\varphi_{n-1},\qquad
  (\Lie_\vv\theta_S)(\vv)=0 .
\]
Hence $S$ solves the $S$-system~\eqref{eq:Ssys} if and only if
\begin{equation}\label{eq:Lietheta}
  \Lie_\vv\theta_S=(\sigma+\varphi_{n-1})\,\theta_S .
\end{equation}
\end{lemma}
\begin{proof}
In coordinates $(\Lie_\vv\theta)_i=\vv(\theta_i)+\sum_j\theta_j\,\partial_i\vv^j$.
For $i=z_k$ the only nonzero $\partial_{z_k}\vv^j$ are $\partial_{z_k}\vv^{z_{k-1}}=1$
$(k\ge1)$ and $\partial_{z_k}\vv^{z_{n-1}}=\varphi_k$, which gives the first two
formulae.  The last is
$(\Lie_\vv\theta)(\vv)=\vv(\theta(\vv))-\theta([\vv,\vv])=0$.  Two $1$-forms that
vanish on $\vv$ and have the same $z$-components are equal, and
$(\sigma+\varphi_{n-1})\theta_S$ has $z$-components
$(\sigma+\varphi_{n-1})S_k$; comparing with~\eqref{eq:Ssys} gives the
equivalence.
\end{proof}

\begin{theorem}[Realisability]\label{thm:frobenius}
For meromorphic $(S_0,\dots,S_{n-2})$ the following are equivalent, locally
on the set where the $S_k$ are regular:
\begin{enumerate}[label=\textup{(\alph*)},nosep,leftmargin=*]
\item $S$ is realisable;
\item $\theta_S\wedge\mathrm{d}\theta_S=0$;
\item $S$ solves the $S$-system~\eqref{eq:Ssys}, and the $\binom n3$ components
$(\theta_S\wedge\mathrm{d}\theta_S)_{z_az_bz_c}$, $0\le a<b<c\le n-1$, vanish.
\end{enumerate}
\end{theorem}
\begin{proof}
(a)$\Rightarrow$(b): $\theta_S=\mathrm{d}I/R$, so
$\mathrm{d}\theta_S=\mathrm{d}(1/R)\wedge\mathrm{d}I$ and
$\theta_S\wedge\mathrm{d}\theta_S=0$.
(b)$\Rightarrow$(a): by Frobenius, locally $\theta_S=f\,\mathrm{d}I$ with
$f\ne0$; then $\D(I)=f^{-1}\theta_S(\vv)=0$ and the $\mathrm{d}z_{n-1}$-component
gives $fI_{z_{n-1}}=1$, so $I_{z_{n-1}}\ne0$ and $S_k=I_{z_k}/I_{z_{n-1}}$.
(a)$\Rightarrow$(c) by Proposition~\ref{prop:Ssys} and (b).
(c)$\Rightarrow$(b): put $\Omega=\theta_S\wedge\mathrm{d}\theta_S$ and
$g=\sigma+\varphi_{n-1}$.  Since $\theta_S(\vv)=0$, Cartan's formula and
Lemma~\ref{lem:Lie} give
\[
  \iota_\vv\Omega=-\theta_S\wedge\iota_\vv\mathrm{d}\theta_S
  =-\theta_S\wedge\Lie_\vv\theta_S=-g\,\theta_S\wedge\theta_S=0 .
\]
Write $\Omega=\beta+\mathrm{d}x\wedge\alpha$ with $\alpha,\beta$ free of
$\mathrm{d}x$, and $\vv=\partial_x+\vv'$.  The $\mathrm{d}x$-free part of
$\iota_\vv\Omega=0$ is $\alpha+\iota_{\vv'}\beta=0$, so $\Omega$ is determined
by $\beta$, whose components are the $(\theta_S\wedge\mathrm{d}\theta_S)_{z_az_bz_c}$.
\end{proof}

The identity $\iota_\vv\Omega=-\theta_S\wedge(\Lie_\vv\theta_S-g\theta_S)$
holds for every $S$, and the $z_{n-1}$-component of $\Lie_\vv\theta_S-g\theta_S$
vanishes while that of $\theta_S$ is $1$; so $\iota_\vv\Omega=0$ is in fact
\emph{equivalent} to the $S$-system.  Thus the Frobenius $3$-form splits into
a part annihilated by $\vv$, which is the $S$-system, and $\binom n3$ basic
components, which are the genuine realisability conditions.  These are
polynomials in the components $S_k$ and their partial derivatives, hence, by
Theorem~\ref{thm:collapse}, in $\sigma$ and its partial derivatives; they cost
no integration.  For $n=2$ there are none.  At bottom
Theorem~\ref{thm:frobenius} is the Frobenius theorem applied to the
projectivised adjoint system; its content is the splitting.

\begin{corollary}\label{cor:n2}
For $n=2$ every solution of the $S$-equation is realisable.  In fact, for
arbitrary $\varphi$ and $S$,
$\theta_S\wedge\mathrm{d}\theta_S=\bigl(\D(S)-S^2-\varphi_zS+\varphi_y\bigr)\,
\mathrm{d}x\wedge\mathrm{d}y\wedge\mathrm{d}z$.
\end{corollary}
\begin{proof}
The first statement is Theorem~\ref{thm:frobenius} with $\binom23=0$.  The
identity follows from $\iota_\vv\Omega=-\theta_S\wedge(\Lie_\vv\theta_S-g\theta_S)$
evaluated on $(\partial_y,\partial_z)$, or by direct computation.
\end{proof}

The next result describes both the solutions and the realisable ones
intrinsically.  Let $\pi:U\to N$, $(x,I_1,\dots,I_n)\mapsto(I_1,\dots,I_n)$, be
the projection to the local space of leaves $N$, an open set of
$\mathbb{C}^n$; functions on $N$ are the elements of $\CC$.

\begin{theorem}[Intrinsic description]\label{thm:intrinsic}
\begin{enumerate}[label=\textup{(\roman*)},nosep,leftmargin=*]
\item $S$ solves the $S$-system if and only if
$\theta_S=h\,\pi^*\omega$ for a meromorphic function $h\ne0$ on $U$ and a
$1$-form $\omega=\sum_ic_i\,\mathrm{d}I_i$ on $N$ ($c_i\in\CC$).  The class of
$\omega$ up to multiplication by $\CC^\times$ is the hyperplane $H$ of
Theorem~\ref{thm:plucker}; thus the solutions of the $S$-system are the
hyperplane fields on $N$.
\item $S$ is realisable if and only if $\omega\wedge\mathrm{d}\omega=0$: the
realisable solutions are the integrable hyperplane fields on $N$.
\item The characteristics $Q\in\VV$ are the vector fields on $N$: the
prolonged evolutionary field of $Q$ maps a first integral $I$ to $\Pi(I,Q)$,
so $Q$ induces the vector field $X_Q$ on $N$ with
$\mathrm{d}I_i(X_Q)=\Pi(I_i,Q)$, and $\PP(\VV)$ is the set of line fields on
$N$.  The hyperplane of $S$ is
$\{Q:\omega(X_Q)=0\}$.
\end{enumerate}
\end{theorem}
\begin{proof}
(i) In the coordinates $(x,I)$, $\vv=\partial_x$ and $\theta_S(\vv)=0$ means
$\theta_S=\sum_i\theta_i\,\mathrm{d}I_i$.  By Lemma~\ref{lem:Lie}, $S$ solves
the system iff $\partial_x\theta_i=g\,\theta_i$ for all $i$.  Choosing $i_0$
with $\theta_{i_0}\not\equiv0$ and putting $h:=\theta_{i_0}$,
$c_i:=\theta_i/\theta_{i_0}$, this says $\partial_xc_i=0$, i.e.\ $c_i\in\CC$.
Conversely, let $\omega=\sum_ic_i\,\mathrm{d}I_i\ne0$.  The $z$-part of
$\pi^*\omega$ is $w=V^\top c$, which satisfies $\D(w)=-B^\top w$; if
$w_{n-1}\equiv0$, the descent in the proof of Proposition~\ref{prop:Rnonzero}
gives $w\equiv0$, hence $c=0$.  So $w_{n-1}\not\equiv0$, and
$\theta:=\pi^*\omega/w_{n-1}$ is some $\theta_S$ with
$\Lie_\vv\theta=-(\partial_xw_{n-1}/w_{n-1})\,\theta$, so Lemma~\ref{lem:Lie}
applies.  In the original
coordinates $\pi^*\omega=\sum_ic_i\nabla I_i$, whose $z$-part is
$V^\top c$, so the map $p$ of the proof of Theorem~\ref{thm:plucker} is, up
to a factor, $Q\mapsto\sum_ic_i\Pi(I_i,Q)$, and its kernel is $H$.
(ii) $\theta_S\wedge\mathrm{d}\theta_S=h^2\,\pi^*(\omega\wedge\mathrm{d}\omega)$,
and $\pi^*$ is injective on forms.
(iii) $\Pi(I_i,Q)\in\CC$ by~\eqref{eq:pairing}, and by
Lemma~\ref{lem:dimV} the map $Q\mapsto(\Pi(I_i,Q))_i$ is a $\CC$-linear
isomorphism $\VV\to\CC^n$, i.e.\ onto the vector fields on $N$.  The last
statement is $\sum_ic_i\Pi(I_i,Q)=\omega(X_Q)$.
\end{proof}

\begin{corollary}[Strictness for every $n\ge3$]\label{cor:strict}
For $n\ge3$ there are solutions of the $S$-system, and solutions $\sigma$ of
$\Ef_n(\sigma)=0$, that are not realisable on any open set.
\end{corollary}
\begin{proof}
Take distinct $a,b,c$ and $\omega=\mathrm{d}I_a+I_b\,\mathrm{d}I_c$.  By
Theorem~\ref{thm:intrinsic}(i) and its proof, its normalisation is a solution
of the $S$-system, and its $\sigma$ solves
$\Ef_n=0$ by Theorem~\ref{thm:collapse}.  But
$\omega\wedge\mathrm{d}\omega=\mathrm{d}I_a\wedge\mathrm{d}I_b\wedge\mathrm{d}I_c\ne0$.
\end{proof}

\begin{corollary}[The two accidents of order two]\label{cor:accidents}
When $n=2$, $N$ is a surface.  Every $1$-form on a surface is integrable, which
is Corollary~\ref{cor:n2}; and a line field on a surface determines a unique
hyperplane field (its annihilator), which identifies $\PP(\VV)$ with
$\PP(\VV^\vee)$ and is the identity $\lambda=-S$
(Proposition~\ref{prop:lambda-dict}).  For $n\ge3$ neither holds: hyperplane
fields need not be integrable (Corollary~\ref{cor:strict}), and a hyperplane
contains a $\PP^{n-2}(\CC)$ of line fields.
\end{corollary}

The novelty relative to~\cite{AncoBluman1998,Mohanasubha2016} is that the
condition is \emph{projective}: an adjoint symmetry $R$ yields an integrating
factor only after multiplication by a suitable element of $\CC$, and
Theorem~\ref{thm:frobenius} decides whether such an element exists without
looking for it.

\section{Consequences for the \texorpdfstring{$S$}{S}-function method}\label{sec:method}

\paragraph{One unknown function.}  By Theorem~\ref{thm:collapse}, a search
for a rational $S$-vector by a bounded-degree ansatz is a search for the
single rational function $\sigma$; the other components follow by
differentiation.  By Theorem~\ref{thm:collapse}(iii) this is the same as
searching for an adjoint symmetry $R$ with rational logarithmic derivative
$\D(R)/R=-(\sigma+\varphi_{n-1})$.  The price is that $\Ef_n$ has order $n-1$
and degree $n$, and that $\widehat S_k$ has degree $n-1-k$ in $\sigma$, so the
degrees of the implied components grow; Table~\ref{tab:growth} records the
size of $\Ef_n$.  Whether the reduced system is cheaper than the
$(n-1)$-component one is a practical question that we do not settle.

\paragraph{Candidates must be tested.}  For $n\ge3$ a solution of
$\Ef_n(\sigma)=0$ is only a candidate: by Theorem~\ref{thm:frobenius} it must
satisfy the $\binom n3$ further conditions, which are polynomial in $\sigma$
and its partial derivatives and cost no integration.  Example~\ref{ex:nonrealisable}
shows that the test can fail.

\paragraph{Recovering the first integral.}  Given a realisable $\sigma$, one
needs $R$ with $\D(R)=-(\sigma+\varphi_{n-1})R$, and then the representative
$cR$, $c\in\CC^\times$, that makes
$\omega_{cR}=cR\,\theta_{\widehat S}$ closed; then $I$ follows by one
quadrature.  The first step is a linear first-order equation along the flow,
which is a quadrature only when a function with the required logarithmic
derivative is known (for instance a hyperexponential $R$ found by a Darboux
search, as in~\cite{ACDM2019,DEM2024}).  The second step is an
integrating-factor problem in the leaf variables.  Both are elementary in
the examples below, because first integrals were already at hand; they are
not quadratures in general: on the second branch
of Example~\ref{ex:search}, $\sigma+\varphi_2=0$, so $R$ is itself a first
integral, and the naive choice $R=1$ does not close the form while
$R=I_1^{-2}$ does.  The ambiguity in $R$ is exactly the reparametrisation
freedom $I\mapsto F(I)$: two first integrals have the same $S$-vector iff
they are functionally dependent, since then $\mathrm{d}I_1\wedge\mathrm{d}I_2=0$.

\paragraph{$\lambda$-symmetries.}  Relative to the evolutionary field
$\partial_y$, $\lambda$ is a $\lambda$-symmetry when
$(\D+\lambda)^n(1)=\sum_k\varphi_k(\D+\lambda)^k(1)$~\cite{MurielRomero2001}.

\begin{proposition}[Dictionary]\label{prop:lambda-dict}
For nonvanishing $Q$ and $\lambda:=\D(Q)/Q$ one has
$(\D+\lambda)^k(1)=\D^k(Q)/Q$ for all $k\ge0$, so the defect of the
$\lambda$-symmetry condition equals $L(Q)/Q$.  Hence $Q\mapsto\D(Q)/Q$ is a
bijection from $\PP(\VV)$ to the $\lambda$-symmetries relative to
$\partial_y$.
\end{proposition}
\begin{proof}
Induction:
$(\D+\lambda)(\D^kQ/Q)=\D^{k+1}Q/Q-\D^kQ\,\D Q/Q^2+(\D Q/Q)(\D^kQ/Q)=\D^{k+1}Q/Q$.
$\lambda$ is unchanged under $Q\mapsto cQ$, $c\in\CC^\times$, and determines
$Q$ up to such a factor.  For surjectivity, let $\lambda$ be a
$\lambda$-symmetry and $\Lambda_k:=(\D+\lambda)^k(1)$.  The defining condition
says $\D\Lambda=B\Lambda-\lambda\Lambda$, so $\D(V\Lambda)=-\lambda V\Lambda$.
Pick $j_0$ with $(V\Lambda)_{j_0}\not\equiv0$ (possible since $V$ is invertible
and $\Lambda_0=1$) and put $Q:=1/(V\Lambda)_{j_0}$, which is meromorphic with
$\D(Q)/Q=\lambda$; then $Q\in\VV$ by the first part.
\end{proof}

Restricting to $\partial_y$ loses nothing: every $\lambda$-symmetry is
equivalent to one of this form~\cite{MurielRomero2001}.

Combined with Theorem~\ref{thm:plucker} and Corollary~\ref{cor:accidents}:
one $\lambda$-symmetry determines no $S$-vector when $n\ge3$; $n-1$
$\lambda$-symmetries whose characteristics are $\CC$-independent determine
one solution of the $S$-system, which need not be realisable; and a solution
determines the hyperplane of compatible $\lambda$-symmetries but no
distinguished member.  The annihilation relation~\eqref{eq:annihilate} for a
single $\lambda$-symmetry is~\cite[Thm.~2.3]{Mohanasubha2016}, and the
non-uniqueness for $n=3$ was observed in~\cite{Mohanasubha2015}; the count
$n-1$ and the realisability caveat are what Theorem~\ref{thm:plucker} adds.
Whether $n-1$ $\lambda$-symmetries give an $(n-1)$-fold reduction by
quadratures is a different question~\cite{MurielRomeroRuiz2017}.

\paragraph{Classical identities in this language.}  For $n$ independent first
integrals, factoring $R_i$ out of the rows of $V=[\partial I_i/\partial z_k]$
gives
\[
  \frac{\partial(I_1,\dots,I_n)}{\partial(z_0,\dots,z_{n-1})}
  =\Bigl(\prod_iR_i\Bigr)\det\mathbb{S},\qquad
  \mathbb{S}_{ik}=S_{i,k},\ S_{i,n-1}=1,
\]
which is Jacobi's last multiplier ($\D(\det V)=-\varphi_{n-1}\det V$ since
$\D V=-VB$ and $\operatorname{tr}B=\varphi_{n-1}$;
see~\cite{NucciLeach2008}); for $n=2$ it is $R_1R_2(S_1-S_2)$.  The pairing
$\Pi(I,Q)$ of~\eqref{eq:pairing} produces a first integral from $I$ and $Q$
without integration and vanishes exactly on the hyperplane of $I$.  It
extracts only one $\CC$-coordinate of $Q$: $\Pi(I,\cdot)$ is a $\CC$-linear
form on $\VV$ with kernel $H$, and for a full set $I_1,\dots,I_n$ the top
entries $Q^*_j$ of the columns of $V^{-1}$ form a basis with
$\Pi(I_i,Q^*_j)=\delta_{ij}$.  Consequently, if $Q=\sum_jc_jQ^*_j$ with the
$c_j$ functions of first integrals already known, $\Pi(I,Q)$ is a function of
those first integrals and yields nothing new.

\section{Examples}\label{sec:examples}

\begin{example}[a third-order equation]\label{ex:third-order}
Let
\begin{equation}\label{eq:thirdorder}
  y''' = \frac{(y'')^2}{y'},\qquad
  \varphi=\frac{z_2^2}{z_1},\quad\varphi_0=0,\quad
  \varphi_1=-\frac{z_2^2}{z_1^2},\quad\varphi_2=\frac{2z_2}{z_1}.
\end{equation}
Independent first integrals and their data are
\[
\begin{aligned}
  I_1&=\frac{z_2}{z_1}, & R_1&=\frac1{z_1}, & S_1&=\Bigl(0,\ -\frac{z_2}{z_1}\Bigr),\\
  I_2&=\ln z_1-\frac{xz_2}{z_1}, & R_2&=-\frac x{z_1}, &
     S_2&=\Bigl(0,\ -\frac{z_2}{z_1}-\frac1x\Bigr),\\
  I_3&=z_0-\frac{z_1^2}{z_2}, & R_3&=\frac{z_1^2}{z_2^2}, &
     S_3&=\Bigl(\frac{z_2^2}{z_1^2},\ -\frac{2z_2}{z_1}\Bigr),
\end{aligned}
\]
and $(\prod R_i)\det\mathbb{S}=-1/z_1^2$ is the Jacobian, a last multiplier.
The variational equation has the $\CC$-independent solutions
$Q_1=1$, $Q_2=-z_1$ (the translations $\partial_y$, $\partial_x$) and
$Q_3=xz_1^2/z_2-z_1^3/z_2^2$; the scaling characteristic $z_0=I_3Q_1-I_1^{-1}Q_2$
is $\CC$-dependent on them although $\mathbb{C}$-independent.  The Plücker
formula~\eqref{eq:plucker} sends $\{Q_1,Q_2\}$, $\{Q_1,Q_3\}$, $\{Q_2,Q_3\}$ to
$S_1$, $S_2$, $S_3$.  This equation is linearisable
($(\ln y'')'=(\ln y')'$), and we use it only because everything can be checked
by hand.
\end{example}

\begin{example}[a search]\label{ex:search}
For~\eqref{eq:thirdorder} take
$\sigma=(a_0+a_1x+a_2z_0+a_3z_1+a_4z_2)/(b_0+b_1x+b_2z_0+b_3z_1+b_4z_2)$
and impose~\eqref{eq:E3}.  Clearing denominators gives $95$ polynomial
equations in the ten coefficients, whose solutions give exactly two
admissible branches, $\sigma=-z_2/z_1$ and $\sigma=-2z_2/z_1$: the principal
components of $S_1$ and $S_3$.  ($S_2$ has $\sigma=-(xz_2+z_1)/(xz_1)$, outside
the ansatz.)  Both pass the realisability test.  On the first branch
$\widehat S_0=0$, $R=1/z_1$ and the recovered integral is $I_1$; on the second
$\widehat S_0=z_2^2/z_1^2$, $R=z_1^2/z_2^2$ and the integral is $I_3$, with
the caveat on $R$ discussed in Section~\ref{sec:method}.
\end{example}

\begin{example}[a solution that is no $S$-vector, $n=3$]\label{ex:nonrealisable}
On~\eqref{eq:thirdorder}, $1-z_2=Q_1+I_1Q_2\in\VV$, and the hyperplane
$\langle1-z_2,Q_3\rangle_\CC$ has $m_2=z_1(x-z_1)$.  Formula~\eqref{eq:plucker}
gives
\begin{equation}\label{eq:bad-S}
  S_0 = \frac{-z_2^2}{z_1(x-z_1)},\qquad
  \sigma = \frac{2z_1z_2-xz_2-z_1}{z_1(x-z_1)} ,
\end{equation}
which solves the $S$-system and~\eqref{eq:E3}, while
\[
  \bigl(\theta_S\wedge\mathrm{d}\theta_S\bigr)_{z_0z_1z_2}
   = \frac{2z_2}{z_1(x-z_1)^2}\ne0 .
\]
In the language of Theorem~\ref{thm:intrinsic}, $\theta_S$ is the normalisation
of $\mathrm{d}I_2+I_1^2\,\mathrm{d}I_3$, an instance of the construction of
Corollary~\ref{cor:strict}.
\end{example}

\begin{example}[$n=4$]\label{ex:n4}
For $y''''=(y''')^2/y''$ the functions $J_1=z_3/z_2$,
$J_2=\ln z_2-xz_3/z_2$ and $J_3=z_1-z_2^2/z_3$ are first integrals, and
$\omega=\mathrm{d}J_2+J_1^2\,\mathrm{d}J_3$ gives the solution of the
$S$-system
\[
  S=\Bigl(0,\ -\frac{z_3^2}{z_2(x-z_2)},\
     -\frac{xz_3-2z_2z_3+z_2}{z_2(x-z_2)}\Bigr),
  \qquad
  \bigl(\theta_S\wedge\mathrm{d}\theta_S\bigr)_{z_1z_2z_3}
   =\frac{2z_3}{z_2(x-z_2)^2}\ne0 .
\]
Four of the ten components of $\theta_S\wedge\mathrm{d}\theta_S$ are nonzero.
\end{example}

\begin{example}[the third-order Emden equation]\label{ex:emden3}
The substitution $y=u'/u$ carries $u''''=0$ into
$y'''+4yy''+3(y')^2+6y^2y'+y^4=0$.  With $p_1=z_0$, $p_2=z_1+z_0^2$,
$p_3=z_2+3z_0z_1+z_0^3$ (so $p_k=u^{(k)}/u$),
\[
  I_1=\frac{p_2-xp_3}{p_3},\quad
  I_2=\frac{p_1-xp_2+\frac12x^2p_3}{p_3},\quad
  I_3=\frac{1-xp_1+\frac12x^2p_2-\frac16x^3p_3}{p_3}
\]
are first integrals.  Their $S$-vectors solve~\eqref{eq:Ssys} and are
realisable; their principal components solve~\eqref{eq:E3}; and
$\partial(I_1,I_2,I_3)/\partial(z_0,z_1,z_2)=p_3^{-4}$.
\end{example}

\paragraph{Verification.}  Every identity and example of this paper was checked
by two independent implementations, in SymPy and in Maple, with no numerical
substitution: the $n=2$ facts of Section~\ref{ssec:two} on the harmonic
oscillator, the modified Emden equation and $y''=y'^2/y$;
Example~\ref{ex:third-order} completely; Theorem~\ref{thm:collapse}(iii) for
$2\le n\le5$ and Proposition~\ref{prop:shape} for $2\le n\le7$;
Example~\ref{ex:search}; Corollary~\ref{cor:n2} for generic $\varphi$ and $S$;
Examples~\ref{ex:nonrealisable}--\ref{ex:emden3}; and the Plücker formula on
$y'''=y''$, $y''''=y'''$ and $y^{(5)}=y'''$ for every $(n-1)$-subset of the
obvious fundamental sets.  The scripts are supplementary material.

\section{Outlook}\label{sec:outlook}

\paragraph{Two base fields.}  The structures of this paper live over the field
$\CC$ of local first integrals.  For rational equations there is also the
algebraic setting $K=\mathbb{C}(x,z_0,\dots,z_{n-1})$ with the derivation
$\D$, whose constants are the \emph{rational} first integrals; for an
equation without them the constants are $\mathbb{C}$, and Picard--Vessiot
theory applies to $L^*$ over $(K,\D)$; we assume this case here.  A rational
solution of the $S$-system then corresponds to a hyperexponential solution of
$L^*$ over $K$, and its existence is governed by the differential Galois group
of $L^*$.  Unlike the case of $\mathbb{C}(x)$, no general algorithm is known
for hyperexponential solutions over $(K,\D)$: finding them amounts to finding
Darboux polynomials of $\D$, which runs into the degree-bound problem of the
Prelle--Singer method.

\paragraph{Compatibility conditions.}  We have not compared the $\binom n3$
conditions of Theorem~\ref{thm:frobenius} with the compatibility constraints
imposed on the null forms in~\cite[eqs.~(2.10)--(2.13)]{CSL2005n}; the
comparison would be useful.

\paragraph{Realisability inside the ansatz.}  The $\binom n3$ conditions of
Theorem~\ref{thm:frobenius} can be imposed together with $\Ef_n(\sigma)=0$ in
a rational search, which would remove non-realisable candidates before any
integration; whether this makes the search cheaper is open.

\paragraph{Partial integrability.}  Knowing $r<n$ independent first integrals
amounts to an $r$-dimensional subspace of $\VV^\vee$, i.e.\ to an integrable
codimension-$r$ distribution on $N$; the corresponding ``partial
$S$-function'' is the analogue of the objects of this paper for that
distribution.

\section*{Statements and declarations}

\noindent\textbf{Funding.} \todo{CNPq/FAPERJ numbers, as in the companion papers}

\medskip\noindent\textbf{Competing interests.} The authors declare none.

\medskip\noindent\textbf{Use of AI tools.}  Claude (Anthropic) was used in
working sessions to propose and check constructions and proofs, to write the
SymPy and Maple verification scripts, to act as an internal referee of
successive drafts, and to draft parts of the text.  Every statement and proof
was checked by the authors, and every computation reported here was
reproduced by two independent implementations.  The AI system is not an
author; the authors take full responsibility for the content.

\medskip\noindent\textbf{Data and code availability.}  The SymPy and Maple
scripts that reproduce every check listed in Section~\ref{sec:examples} are
provided as supplementary material.

\section*{Acknowledgements}

We thank Claude (Anthropic), the AI assistant described above, for its help in
developing this work.

\appendix
\section{The closing equations}\label{app:closing}

\begin{proposition}\label{prop:shape}
For every $n\ge2$:
\begin{enumerate}[label=\textup{(\roman*)},leftmargin=*,nosep]
\item $\widehat S_k$ has degree $n-1-k$ in $\sigma$ with leading coefficient
$1$; in particular $\Ef_n$ is monic of degree $n$;
\item the coefficient of $\sigma^{n-1}$ in $\Ef_n$ is $(n-1)\varphi_{n-1}$;
\item $\D^{n-1}\sigma$ occurs linearly in $\Ef_n$, with coefficient
$(-1)^{n-1}$;
\item for $n\ge3$ the coefficient of $\sigma^{n-2}$ is
$-\varphi_{n-2}+\binom{n-1}{2}\bigl(\varphi_{n-1}^2-\D\varphi_{n-1}\bigr)
-\binom{n}{2}\D\sigma$.
\end{enumerate}
\end{proposition}
\begin{proof}
Write $\widehat S_k=\sigma^{n-1-k}+a_k\sigma^{n-2-k}+b_k\sigma^{n-3-k}+\cdots$,
with $a_{n-2}=b_{n-2}=0$.  (i) In~\eqref{eq:recursion} the top power comes
only from $\sigma\widehat S_k$.  (ii) Comparing coefficients of
$\sigma^{n-1-k}$, $a_{k-1}=a_k+\varphi_{n-1}$ for $k\le n-2$, so
$a_{-1}=(n-1)\varphi_{n-1}$.  (iii) Only $-\D(\widehat S_k)$ raises the order,
by one and with a sign, at each of the $n-1$ steps after
$\widehat S_{n-2}=\sigma$.  (iv) Comparing coefficients of $\sigma^{n-2-k}$,
$b_{k-1}=b_k+\varphi_{n-1}a_k-(n-1-k)\D\sigma-\D(a_k)$ with
$a_k=(n-2-k)\varphi_{n-1}$; summing gives
$\binom{n-1}{2}(\varphi_{n-1}^2-\D\varphi_{n-1})-\binom n2\D\sigma$.  The
inhomogeneous term $-\varphi_k$ reaches degree at most $k$ in $\Ef_n$, so only
$-\varphi_{n-2}$ contributes to degree $n-2$.
\end{proof}

For $n=2$, $\Ef_2=\sigma^2+\varphi_1\sigma-\varphi_0-\D\sigma$.  For $n=3$,
$\Ef_3=0$ is~\eqref{eq:E3}.  For $n=4$,
\[
\begin{aligned}
  \Ef_4={}&\sigma^4+3\varphi_3\sigma^3
  +\bigl[-\varphi_2+3\varphi_3^2-3\D\varphi_3-6\D\sigma\bigr]\sigma^2\\
  &+\bigl[-\varphi_1+2\D\varphi_2+\varphi_3^3+\D^2\varphi_3
        +4\D^2\sigma-2\varphi_2\varphi_3-3\varphi_3\D\varphi_3
        -9\varphi_3\D\sigma\bigr]\sigma\\
  &-\varphi_0+\D\varphi_1-\D^2\varphi_2+3(\D\sigma)^2
        -\D^3\sigma-\varphi_1\varphi_3-\varphi_2\varphi_3^2
        +\varphi_2\D\varphi_3+\varphi_2\D\sigma\\
  &-3\varphi_3^2\D\sigma+2\varphi_3\D\varphi_2
        +3\varphi_3\D^2\sigma+3\D\varphi_3\,\D\sigma .
\end{aligned}
\]

\begin{table}[htbp]
\centering
\renewcommand{\arraystretch}{1.15}
\begin{tabular}{@{}cccc@{}}
\toprule
$n$ & degree in $\sigma$ & differential order & expanded terms \\
\midrule
2 & 2 & 1 & 4 \\
3 & 3 & 2 & 11 \\
4 & 4 & 3 & 27 \\
5 & 5 & 4 & 59 \\
6 & 6 & 5 & 122 \\
7 & 7 & 6 & 237 \\
\bottomrule
\end{tabular}
\caption{Size of the closing equation $\Ef_n$, with $\sigma$ and the
$\varphi_k$ treated as independent differential indeterminates.}
\label{tab:growth}
\end{table}


\end{document}